\documentclass[11=pt]{amsart}
\usepackage[left = 3.0cm,right = 3.0cm, top= 3.0cm, bottom = 4.0cm]{geometry}
\usepackage{amsthm,amsmath,amssymb,amscd,graphics,enumerate, stmaryrd,xspace,verbatim, epic, eepic,url}
\usepackage[usenames,dvipsnames]{color}
\usepackage{mathtools}
\usepackage[colorlinks=true]{hyperref}
\usepackage[skip=5pt, indent = 10pt]{parskip}

\usepackage[active]{srcltx}
\usepackage{mathrsfs}
\usepackage{tikz}
\usepackage{tikz-cd}

\usetikzlibrary{shapes,shadows,arrows,patterns,trees,automata,positioning,fadings,backgrounds,intersections,calc,decorations.markings}
\tikzstyle{vertex} = [outer color=black,draw, color=black,line width=0.2mm, inner color=black, circle,inner sep=0.2mm,minimum size=0.2mm] 
\tikzstyle{edge} = [line width = 0.3mm] 

\makeatletter
\DeclareRobustCommand{\cev}[1]{%
  {\mathpalette\do@cev{#1}}%
}
\newcommand{\do@cev}[2]{%
  \vbox{\offinterlineskip
    \sbox\z@{$\m@th#1 x$}%
    \ialign{##\cr
      \hidewidth\reflectbox{$\m@th#1\vec{}\mkern4mu$}\hidewidth\cr
      \noalign{\kern-\ht\z@}
      $\m@th#1#2$\cr
    }%
  }%
}
\makeatother

\def\:{\colon}
\def\.{,\dots,}

\def\Mgn{\overline{\mathcal{M}}_{g,n}}
\def\DMgn{\overline{\mathcal{M}}_{g,A}}

\def\multiset#1#2{\ensuremath{\left(\kern-.3em\left(\genfrac{}{}{0pt}{}{#1}{#2}\right)\kern-.3em\right)}}

\newcommand{\SH}{\mathcal{S}\mathsf{H}}

\newcommand{\Mdiamond}{{\mathcal{M}_{g,A}^\diamond}}

{
   \newtheorem{theorem}[subsection]{Theorem}
      \newtheorem*{theorem*}{Theorem}

   \newtheorem{lemma}[subsection]{Lemma}

   \newtheorem{corollary}[subsection]{Corollary}
   \newtheorem{conjecture}[subsection]{Conjecture}
   \newtheorem*{conjecture*}{Conjecture}

}
{\theoremstyle{definition}
          \newtheorem*{exercise*}{Exercise}
   
   \newtheorem{example}[subsection]{Example}
   \newtheorem*{example*}{Example}
   \newtheorem{definition}[subsection]{Definition}
   
   \newtheorem*{definition*}{Definition}
   
   \newtheorem{remark}[subsection]{Remark}

\title{Pullbacks of Brill-Noether Classes under Abel-Jacobi Sections}

\author{Sam Molcho \vspace{-3em}}

\begin{document}
\bibliographystyle{amsalpha}
\maketitle

\begin{abstract} 
We prove that the pullbacks of the virtual fundamental classes of the Brill-Noether loci under any Abel-Jacobi section lie in the tautological ring of $\overline{\mathcal{M}}_{g,n}$. This resolves a conjecture of \cite{PRvZ}, and is part of a broader program to understand the logarithmic intersection theory of $\Mgn$. 
\end{abstract}

\section{Introduction}
Let $\mathcal{M}_{g,n}$ be the moduli space of smooth genus $g$, $n$-pointed curves, and $\Mgn$ its compactification by stable curves. Many of the deepest results on the geometry of $\Mgn$ are tied to its enumerative geometry: the study of its Chow ring $\mathsf{CH}^\star(\Mgn)$. While it is currently difficult to get a handle on the full Chow ring, Mumford introduced in \cite{MumTow} a distinguished subring 
$$
R^\star(\Mgn) \subset \mathsf{CH}^\star(\Mgn)
$$
called the \emph{tautological} ring of $\Mgn$. The name is due to the definition of $R^\star(\Mgn)$ as generated by certain classes on $\Mgn$ that arise from its structure as a moduli space: classes that come from the cotangent bundle on the markings and the canonical bundle, and are thus present on every family of curves, without special reference to their geometry. 

Since its introduction in \cite{MumTow}, efforts to understand the structure of the tautological ring have led some of the most striking advances of our understanding of $\Mgn$. The results of these efforts have been substantial: we now know that $R^\star(\Mgn)$ is, surprisingly, intimately connected with Gromov-Witten theory \cite{Witten,KonAiry}. We understand \cite{GrPannonTaut} an additive set of generators of $R(\Mgn)$; and, following work of Faber-Zagier and Pixton, we have a conjectural full set of relations \cite{PPrel,PPZ}. Furthermore, computer software has been developed \cite{admcycles} that allows us to perform calculations which were far out of reach even in recent years. 

The question then that one is immediately confronted with when encountering a geometrically interesting cycle on $\Mgn$ is ``Is it in the tautological ring?". If the answer is affirmative, the immediate follow up is ``What is a formula for the class, in terms of the generators of $R^\star(\Mgn)$?". And in the rare cases where a class naturally has more than one such formula, one can ask ``How does it fit with the Faber-Zagier and Pixton relations?". Our goal in this paper is to answer the first question for certain loci on $\Mgn$ that come from Brill-Noether theory. Forthcoming work with Abreu and Pagani aims to answer the remaining two. 

To explain the source of these cycles, recall that the Jacobian 
$$
\mathsf{Jac}^d(\mathcal{C}_{g,n}/\mathcal{M}_{g,n}) \to \mathcal{M}_{g,n}
$$  
parametrizing degree $d$ line bundles on the universal curve $\mathcal{C}_{g,n}$ is an abelian scheme. Extending $\mathsf{Jac}^d$ to $\Mgn$ is a subtle problem, as the space of all degree $d$ line bundles is neither universally closed nor separated over the boundary of $\Mgn$.  Compactifications of the generalized Jacobian have been the subject of a vast literature, with \cite{OdaSeshadri,Pan,Cap,Est,Melo,KP} being just a few representatives. The best behaved of these compactifications are probably the compactifications constructed in \cite{KP}. The \cite{KP} compactifications depend on the notion of a \emph{universal numerical stability condition} $\theta$. This is a collection of real numbers assigned to every component of every fiber of $\overline{\mathcal{C}}_{g,n} \to \Mgn$, adding up to a fixed degree $d \in \mathbb{Z}$, and which behave appropriately with respect to specializations of curves. The precise notion itself is not relevant to the present paper, but its consequence is that for every choice of $\theta$ of degree $d$ that is sufficiently generic, there is a compactification 
$$
\mathsf{Jac}^d(\mathcal{C}_{g,n}/\mathcal{M}_{g,n}) \subset \mathsf{Pic}_{g,n}^\theta \to \Mgn
$$
which is smooth, and which is representable over $\Mgn$. While it is not always true that a generic $\theta$ can be found, in the presence of a marking, i.e. when $n \ge 1$, generic $\theta$ always exist. For generic $\theta$, $\mathsf{Pic}_{g,n}^\theta$ is in fact a fine moduli space, parametrizing $\theta$-stable line bundles on \emph{quasi-stable} curves, i.e. curves that can have chains of semistable rational components of length at most $1$. In particular, $\mathsf{Pic}^{\theta}_{g,n}$ carries a universal (quasi-stable) curve 
$$
\pi:\mathcal{C}_{g,n}^\theta \to \mathsf{Pic}_{g,n}^\theta 
$$  
and a line bundle $\mathcal{L}$. The Brill-Noether loci in $\mathsf{Pic}_{g,n}^\theta$ are the loci 
$$
\mathcal{W}_{d}^r(\theta): = \left \{(C,L): L \textup{ has at least } r+1 \textup{ sections}  \right \} \subset \mathsf{Pic}_{g,n}^\theta
$$
which extend the usual Brill-Noether loci in $\mathsf{Jac}^d$. The loci $\mathcal{W}_{d}^r(\theta)$ can have components of various dimensions, but they support a virtual fundamental class, an algebraic cycle 
$$
\mathsf{w}_{d}^r(\theta) := \Delta_{g-d+r}^{(r+1)}(-R\pi_*\mathcal{L})
$$
of the expected codimension $g - \rho$, for $\rho$ the Brill-Noether number $g-(r+1)(g-d+r)$, defined as a degeneracy locus of $R\pi_*\mathcal{L}$ in \cite{PRvZ}. 

Fix now an integer $k$, a vector of integers $(a_1,\cdots,a_n)$, and $D$ a boundary divisor of the universal curve $\overline{\mathcal{C}}_{g,n}$ of $\Mgn$. We package this data in a single symbol ``$A$'', and define a rational section 
\begin{align*}
\mathsf{aj}_A: \Mgn \to \mathsf{Pic}_{g,n}^\theta\\
(C,x_1,\cdots,x_n) \mapsto \omega^k(\sum a_ix_i) \otimes \mathcal{O}(D|_C).
\end{align*}
We can then pull back the classes $\mathsf{w}_{d}^r(\theta)$ to $\Mgn$ to get classes 
$$
\mathsf{w}_{g,A,d}^r(\theta):= \mathsf{aj}_A^*(\mathsf{w}_{g,d}^r(\theta)) \in \mathsf{CH}^{g-\rho}(\Mgn)\footnote{We work throughout with Chow groups with rational coefficients.}.
$$
Pagani, Ricolfi and van Zelm conjecture
\begin{conjecture}\cite{PRvZ}
\label{PRvZ conjecture}
The classes $\mathsf{w}_{g,A,d}^r(\theta)$ lie in the tautological ring $R^*(\Mgn)$ for any choice of $g,d,A$ and generic $\theta$.
\end{conjecture}
The main result in this paper is 
\begin{theorem}
\label{thm: main}
Conjecture \ref{PRvZ conjecture} is true. 
\end{theorem}

\subsection{Logarithmic Chow Rings} In fact, we prove conjecture \ref{PRvZ conjecture} by proving a stronger statement, using ideas from logarithmic geometry. The stack $\Mgn$, together with its divisor $\partial \Mgn= \Mgn - \mathcal{M}_{g,n}$ has the structure of a \emph{toroidal embedding}, or, equivalently, of a \emph{logarithmic scheme}. 

Starting with any pair $(X,D)$ of a smooth DM stack with a normal crossings divisor, the divisor $D$ stratifies $X$. A simple blowup is the blowup of $X$ along a smooth stratum closure. Such a blowup $p:X' \to X$ produces a new pair $(X',D' = p^{-1}(D))$. A blowup obtained by iterating this procedure a finite number of times is called an iterated blowup. A \emph{logarithmic modification} of $(X,D)$ is any proper birational map $p:X' \to X$ which can be dominated by an iterated blowup of $(X,D)$. 

Non-singular logarithmic modifications form an inverse system by refinement: a map $X'' \to X'$ belongs to the system if $X'' \to X$ factors through $X' \to X$. This way, we get a system of rings $\mathsf{CH}(X')$ indexed by Gysin pullback. The \emph{logarithmic} Chow ring of $X$ is defined (\cite{BarrottChow},\cite{MPS}) as 
$$
\log\mathsf{CH}(X,D) := \varinjlim \mathsf{CH}(X')
$$
where $X' \to X$ ranges through the non-singular log modifications of $(X,D)$. The ordinary Chow ring $\mathsf{CH}(X)$ is contained in $\mathsf{logCH}(X,D)$ as a subring, and there is a retraction (which is a group but not a ring homomorphism) $\mathsf{logCH}(X,D) \to \mathsf{CH}(X)$ by pushforward. 

The ring $\log\mathsf{CH}(X,D)$ is almost never finitely generated, but any given element $\gamma \in \log\mathsf{CH}(X,D)$ is determined by a finite amount of data: a (typically non-canonical) choice of a log modification $X' \to X$, and a \emph{representative} $\gamma(X') \in \mathsf{CH}(X')$: a class in $\mathsf{CH}(X')$ that maps to $\gamma$ under the natural inclusion $\mathsf{CH}(X') \to \log\mathsf{CH}(X,D)$. 

Moreover, the infinite generation of $\log \mathsf{CH}(X,D)$ is, relative to $\mathsf{CH}(X)$ a purely combinatorial problem. To every pair $(X,D)$ there is an associated complex of rational polyhedral cones $\Sigma_X$, and log modifications $X' \to X$ correspond bijectively to subdivisions $\Sigma_X' \to \Sigma_X$. The Chow ring $\mathsf{CH}(X')$ is then generated by $\mathsf{CH}(X)$ and the algebra $\mathsf{PP}(\Sigma_X')$ of piecewise polynomial functions on $\Sigma_X'$: the continuous functions which are polynomial on each cone of $\Sigma_X'$. The theory is reviewed in \ref{sec: Preliminaries}. Thus, 

$$
\log\mathsf{CH}(X,D) = \left \langle \mathsf{CH}(X), \varinjlim \mathsf{PP}(\Sigma_X') \right \rangle
$$
is generated by $\mathsf{CH}(X)$ and the algebras of piecewise polynomial functions on all subdivisions $\Sigma_X' \to \Sigma_X$.

When applying the construction to $(\Mgn,\partial \Mgn)$, we get a ring 
$$
\log\mathsf{CH}(\Mgn,\partial \Mgn)
$$
which combines the difficulties of the ordinary Chow ring $\mathsf{CH}(\Mgn)$ with the combinatorial complexity of piecewise polynomials on subdivisions of its cone complex $\Sigma_{\Mgn}$. As the complexity of $\mathsf{CH}(\Mgn)$ is ameliorated when restricting to the tautological ring $R^\star(\Mgn)$, we can also restrict to the \emph{logarithmic tautological ring}

$$
\log R^\star(\Mgn,\partial \Mgn) := \left \langle R^\star(\Mgn),\varinjlim \mathsf{PP}(\Sigma_{\Mgn}') \right \rangle \subset \log \mathsf{CH}(\Mgn,\partial \Mgn).
$$

Our proof of conjecture \ref{PRvZ conjecture} proceeds by first lifting the classes $\mathsf{w}_{g,A,d}^r(\theta)$ to $\log \mathsf{CH}(\Mgn)$. We do so by finding, for each generic stability condition $\theta$, a log modification $p:\DMgn^\theta \to \Mgn$ and a lift $\log \mathsf{w}_{g,A,d}^r(\theta)$ of $\mathsf{w}_{g,A,d}^r(\theta)$ to $\mathsf{CH}(\DMgn^\theta)$. The lift is non-trivial, in the sense that  
\begin{align*}
p_*(\log \mathsf{w}_{g,A,d}^r(\theta)) = \mathsf{w}_{g,A,d}^r(\theta) \\
p^*(\mathsf{w}_{g,A,d}^r(\theta)) \neq \log \mathsf{w}_{g,A,d}^r(\theta).
\end{align*}
We then show 

\begin{theorem}
\label{thm: logtaut}
    The class $\log \mathsf{w}_{g,A,d}^r(\theta)$ is in $\log R^*(\Mgn,\partial \Mgn)$.
\end{theorem}
Paradoxically, the stronger theorem \ref{thm: logtaut} is easier to prove than \ref{thm: main}. The reason is that the refined classes $\log \mathsf{w}_{g,A,d}^r(\theta)$ have more structure than the classes $\mathsf{w}_{g,A,d}^r(\theta)$. Our proof here goes by observing that, contrary to $\mathsf{w}_{g,A,d}^r(\theta)$, the class $\log \mathsf{w}_{g,A,d}^r(\theta)$ is a degeneracy locus on $\DMgn^\theta$, and so can be calculated by Grothendieck-Riemann-Roch. We then show that every term that appears in the Grothendieck-Riemann-Roch calculation is either a piecewise polynomial term or a pullback of a tautological class. A complete calculation of this class via the methods of this paper is possible, but is far subtler, and is the subject of the afforementioned forthcoming work with Abreu and Pagani. 

A special case of this construction, when $d=r=0$ and $\theta$ is generic and small (i.e. appropriately close to the $0$ stability condition) yields the \emph{double ramification cycle} $\mathsf{DR}_{g,A}$. In fact the representative of the lift $\log \mathsf{w}_{g,A,0}^0(\theta)$ on $\DMgn^\theta$ is precisely the lift $\mathsf{DR}_{g,A}(\theta)$ of \cite{HMPPS} representing the logarithmic double ramification cycle $\log \mathsf{DR}_{g,A}$ -- which is also a refinement of the double ramification cycle with better functorial properties. As an application, we obtain a fourth proof (after \cite{MRan,HS,HMPPS}) of 

\begin{conjecture}\cite[Conjecture C, slightly reinterpreted]{MPS} \cite[Theorem]{MRan,HS,HMPPS}
    The logarithmic double ramification cycle is in $\log R^*(\Mgn,\partial \Mgn)
    $. 
\end{conjecture}
 While the fact that $\mathsf{DR}_{g,A} \in R(\Mgn)$ has been known for a while (see \cite{JPPZ,FabPan}), the refined version $\mathsf{DR}_{g,A}(\theta) \in \log R^*(\Mgn,\partial \Mgn)$ is rather recent and has several new implications. The reason is that the intersections of the cycles $\mathsf{DR}_{g,A}(\theta)$ relate to the logarithmic Gromov-Witten theory of any toric variety, while $\mathsf{DR}_{g,A}$ only relates to the Gromov-Witten theory of $\mathbb{P}^1$. We do not wish to discuss these implications in detail here, and instead refer the reader to the discussion in \cite{HMPPS,MRan,HS}. We simply mention that the essence of these implications has the following form: the pushforward of the virtual class of the space of (rubber) relative stable maps to a toric variety with respect to its toric boundary to $\Mgn$ can be calculated (\cite{HPS},\cite{RanProduct},\cite{MRan}) as
\begin{equation*}
p_*(\prod_{i} \log \mathsf{DR}_{g,A_i})
\end{equation*}
for appropriate choices of $A_i$, where $p: \log \mathsf{CH}(\Mgn,\partial \Mgn) \to \mathsf{CH}(\Mgn)$
is the retraction. No description as a product of ordinary double ramifiction cycles $\mathsf{DR}_{g,A_i}$ exists. Thus, knowledge that $\log \mathsf{DR}_{g,A}$ is in $\log R^{\star}(\Mgn,\partial\Mgn)$ is enough to conclude that the classes are tautological -- and formulas for the $\log \mathsf{DR}_{g,A}$ provide formulas for the virtual classes-- but knowledge that $\mathsf{DR}_{g,A} \in R^\star(\Mgn)$ alone does not.  \\

 \subsection*{Acknowledgments} I'm grateful to Rahul Pandharipande, Alex Abreu, and Dhruv Ranganathan for several discussions related to the double ramification cycle and the Thom-Porteous formula, and for helpful comments on this paper. Special thanks are due to Nicola Pagani, for introducing me to the problem, countless inexplicably patient explanations on the tautological ring, and a very thorough reading of various preliminary drafts. The paper is all the better for their inputs. I was supported by the grant ERC-2017-AdG-786580-MACI.

\section{Preliminaries}
\label{sec: Preliminaries}
The proof of Theorem \ref{thm: main} makes use of the toroidal structure of the Jacobians $\mathsf{Pic}_{g,n}^\theta$ and $\Mgn$. We collect here the notions and results we need from the theory, but warn the reader that we assume them known and do not attempt to review them. Thorough discussion can be found in \cite{KKMS, K} for foundations on log/toroidal geometry, \cite{AW,CCUW} for foundations on cone complexes, and \cite{MPS,HMPPS,MRan} for the intersection theoretic aspects. The latter sources include a practical review of the foundations along the lines that we use in this paper, so the interested and practically-minded reader can simply start there. 

In this paper we work with toroidal embeddings $(X,D)$: a pair of a variety (or, more generally, algebraic stack) $X$ and a divisor $D$ which, smooth locally looks like a toric variety with its torus invariant divisor. The key example will be a smooth variety $X$ with a normal crossings divisor $D$, but we will need to consider singular toroidal embeddings at some point. Toroidal embeddings have a natural log structure given by the divisor $D$, and via this log structure they can be identified with the \emph{log smooth} log schemes over the base field. We will thus use the term ``toroidal" and ``log smooth" interchangeably, prefering the latter when we do not want to carry the divisor $D$ in the notation. Morphisms of toroidal embeddings for us are the maps that respect this log structure, and we refer to them as log maps. 

We write $\Sigma_X$ for the cone stack with integral structure associated to $(X,D)$, and $\mathcal{A}_X$ for its Artin fan. Recall that $\mathcal{A}_X$ is an Artin stack and $X$ comes with a canonical smooth map 
$$
\alpha_X:X \to \mathcal{A}_X
$$

A map $f:(X,D) \to (Y,E)$ is called log smooth if the associated map of log schemes is log smooth. In practice, this means that the map, which is a priori locally modelled on maps of toric varieties, is modelled on dominant maps of toric varieties. A compact translation of this condition is as follows: after replacing $X$ with a smooth cover, the natural map $\alpha_X \times f$ in the diagram  
\[
\begin{tikzcd}
X \ar[r,"\alpha_X \times f"] \ar[rd,"f"] & \mathcal{A}_X \times_{\mathcal{A}_Y} Y \ar[r] \ar[d] & \mathcal{A}_X \ar[d] \\ & Y \ar[r,"\alpha_Y"] & \mathcal{A}_Y
\end{tikzcd}
\]
is smooth. 

We recall the three elements of the theory that we will need: 

\underline{1. Piecewise Polynomials:} Let $(X,D)$ be a pair of a smooth variety (or DM stack) $X$ with a normal crossings divisor $D$. The algebra of piecewise polynomial function $\mathsf{PP}(\Sigma_X)$ is identified with $\mathsf{CH}(\mathcal{A}_X)$, and thus comes with a graded ring homomorphism \cite{MPS,MRan,HS} 
$$
\mathsf{PP}(\Sigma_X) \to \mathsf{CH}(X).
$$
In this paper, we will adopt the following convention:

\begin{definition}
We write $\mathsf{PP}(X)$ for the image of $\mathsf{PP}(\Sigma_X)$ in $\mathsf{CH}(X)$. 
\end{definition}
We refer to elements of $\mathsf{PP}(X)$ as piecewise polynomials on $X$. The algebra $\mathsf{PP}(X)$ is essentially\footnote{As the closed strata of $(X,D)$ need not be normal, normalization may be required to talk about their normal bundles.} the subalgebra of $\mathsf{CH}(X)$ generated by the strata of $(X,D)$ and the Chern roots of their normal bundles.

\underline{2. Semistable Reduction} The second piece of information we need from $\Sigma_X$ is its relation with \emph{logarithmic alterations}. We have 

\begin{itemize}
\item Subdivisions of $\Sigma_X$ correspond to logarithmic modifications $X' \to X$. These are proper and birational, and are isomorphisms over $X-D$. 
\item Integral substructures of the integral structure correspond to roots $X' \to X$. These are proper, bijective, and of DM-type, and are isomorphisms over $X-D$. 
\end{itemize}

For foundations on roots, the reader may consult \cite{BV,Mss,GMtor}. We mention however that roots do not change the rational Chow groups of $X$, and are only necessary in this paper as a technical crutch, in order to apply the semistable reduction theorem, which should realistically be used as a black box for our purposes. We recall how the correspondence goes: the category of cone stacks is equivalent to the category of Artin fans. Thus, either type of map $\Sigma_X' \to \Sigma_X$ corresponds to a map of Artin fans $\mathcal{A}_X' \to \mathcal{A}_X$. The log modification or root corresponding to the subdivision or integral substructre $\Sigma_X' \to \Sigma_X$ is then 
$$
X' := X \times_{\mathcal{A}_X} \mathcal{A}_X' \to X.
$$

\begin{definition}
A logarithmic alteration $X' \to X$ is a composition of logarithmic modifications and roots. 
\end{definition}

Logarithmic alterations are in particular proper and birational. It is shown in \cite{MPS} that for a logarithmic modification $p: X' \to X$, we have 
$$
p_*\mathsf{PP}(X') = \mathsf{PP}(X).
$$ 
Since roots do not change rational Chow groups, the same result is also true for roots. Thus, we have 

\begin{corollary}
Let $p: X' \to X$ be a logarithmic alteration. Then $p_*\mathsf{PP}(X') = \mathsf{PP}(X)$.
\end{corollary}
 
 We now review the semistable reduction theorem. A log smooth log map $X \to S$ is called \emph{weakly semistable} if it is flat with reduced fibers. The strongest version\footnote{In fact, minor modifications of the significantly simpler results of \cite{AK,Mss} are also sufficient for our purposes.} of semistable reduction is then:    

\begin{theorem}\cite{ALT}[Semistable Reduction]
Let $X \to S$ be a log smooth and surjective map between toroidal embeddings. Then, there exist a log alteration $S' \to S$, and a log alteration $X' \to X \times_S S'$ such that $X' \to S'$ is weakly semistable, and so that $X',S'$ are smooth. 
\end{theorem}

\underline{Twistors:} The third and final tool we need from logarithmic geometry is that of a piecewise linear function. This is in fact an instance of the map between piecewise polynomials and the Chow ring, but it is special enough that we mention it separately. The degree one piecewise polynomials on $\Sigma_X$ are called the piecewise linear functions, and come with a homomorphism 
$$
\mathsf{PL}(\Sigma_X) \to \mathsf{Pic}(X).
$$
Given a piecewise linear function $\alpha$, we write $\mathcal{O}_X(-\alpha)$ for the associated line bundle on $X$. Their importance in our context is that, given a generically smooth family of curves $C \to S$, with $S$ a trait, given their canonical log structures, then the line bundles of the form $\mathcal{O}_C(\alpha)$ are precisely the limits of the trivial bundle on $C_\eta$. Over an arbitrary base $(S,D)$, the line bundles $\mathcal{O}_C(\alpha)$ are the equivalence classes of the Cartier divisors supported on the boundary of $C$. 

\section{Strata Homology}
When the pair $(X,D)$ is toroidal but potentially singular, it is shown in \cite{MRan} that the algebra $\mathsf{PP}(\Sigma_X)$ captures the \emph{operational} Chow ring of the Artin fan $\mathcal{A}_X$, 
$$
\mathsf{PP}(\Sigma_X) \cong \mathsf{CH}^{\textup{op}}(\mathcal{A}_X).
$$
For our present purposes, we mostly need to study $\mathsf{PP}(X)$ from a homological perspective. Thus, we would like to also capture the properties of the ordinary Chow groups of $\mathcal{A}_X$. To that end, we simply define:

\begin{definition}
Let $(X,D)$ be a toroidal embedding, and $\alpha: X \to \mathcal{A}_X$ the map to its Artin fan. The \emph{strata homology} groups of $X$ are 
\[
\SH (X) = \alpha^*\mathsf{CH}(\mathcal{A}_X).
\]
\end{definition}
We currently do not know of a truly pleasant geometric description of $\SH (X)$, along the lines of the normally decorated strata classes of \cite{MPS}, although we find it probable that such a description can be obtained. At the moment all we know are two essentially obvious descriptions, one ``intrinsic", and one ``extrinsic". 

The intrinsic description is obtained by excision. Recall that the Chow groups of a variety that is stratified by affine spaces are generated by the closures of the strata. Of course, since $\mathcal{A}_X$ is a stack, this does not hold anymore, but an intermediate statement can be obtained. Recall that the strata of $X$ are in bijection with the strata of $\mathcal{A}_X$ are in bijection with the strata of $\Sigma_X$. For a cone $\sigma \in \Sigma_X$, we write $\mathcal{X}_\sigma$ for the corresponding stratum. We have 
$$
\mathcal{X}_\sigma = B(\mathbb{G}_m^k \rtimes G)
$$
for some $k$ and some finite subgroup $G \subset \mathfrak{S}_k$ of the symmetric group on $k$ elements. We write $\overline{\mathcal{X}}_\sigma$ for the closure of $\mathcal{X}_\sigma$, and $Z_\sigma = \overline{\mathcal{X}}_\sigma - \mathcal{X}_\sigma$. Then, excision says that the sequence 
\[
\begin{tikzcd}
\mathsf{CH}(Z_\sigma) \ar[r] & \mathsf{CH}(\overline{\mathcal{X}}_\sigma) \ar[r] & \mathsf{CH}(\mathcal{X}_\sigma) \ar[r] & 0
\end{tikzcd}
\]
is exact -- in fact, the argument in \cite[Section 5]{MPS} shows that the sequence is also injective on the left. Since $\mathcal{X}_\sigma$ is smooth, we have that
$$
\mathsf{CH}(\mathcal{X}_\sigma) = \mathsf{P}(\sigma)
$$
is isomorphic to the polynomial functions on $\sigma$. We \emph{choose} a section $s_\sigma: \mathsf{CH}(\mathcal{X}_\sigma) \to \mathsf{CH}(\overline{\mathcal{X}}_\sigma)$ arbitrarily, and find 
$$
\mathsf{CH}(\overline{\mathcal{X}}_\sigma) = \left \langle s_\sigma(\mathsf{P}(\sigma)),(i_{Z_\sigma})_*(\mathsf{CH}(Z_\sigma) \right \rangle.
$$
Inductively, a similar description holds for $Z_\sigma$. The minimal strata are closed, and since they are smooth, their operational Chow groups agree with their Chow groups. Thus, we find
\begin{lemma}
    The Chow group $\mathsf{CH}(\mathcal{A}_X)$ is generated by  
\[(i_{\overline{\mathcal{X}}_\sigma})_*(s_\sigma(\mathsf{P}(\sigma)).
\]
\end{lemma}
The description shows in particular that 
$$
\SH X \subset \oplus_{\sigma \in \Sigma_X} \mathsf{P}(\sigma)
$$
For an element $F \in \mathsf{P}(\sigma)$, the element $s_{\sigma}(F)$ may be thought of as the closure of a cycle representing $F$.  Choosing these cycles makes this description impractical. A more practical extrinsic description can be obtained, this time by choice of a desingularization: 

\begin{lemma}
Let $(X,D)$ be a toroidal embedding, and $p:(X',D') \to (X,D)$ a desingularization by a log alteration. Then 
$$
\SH X = p_*\mathsf{PP}(X').
$$
\end{lemma}

\begin{proof}
We have a Cartesian diagram 
\[
\begin{tikzcd}
X' \ar[r,"\alpha'"] \ar[d,"p"] & \mathcal{A}_X' \ar[d,"q"] \\ 
X \ar[r,"\alpha"] & \mathcal{A}_X
\end{tikzcd}
\]
with smooth horizontal maps. Thus the equality follows from the compatibility 
$$
p_*(\alpha')^* = \alpha^*q_*
$$
and the equality 
$$
(\alpha')^*\mathsf{CH}(\mathcal{A}_X') = \mathsf{PP}(X').
$$
\end{proof}
The same proof in fact suffices shows also: 

\begin{lemma}
\label{cor:pushforward}
Let $p:X' \to X$ be a log alteration. Then 
$$
p_* \SH (X') = \SH (X).
$$
\end{lemma}

For any piecewise polynomial $\mathsf{PP}(\Sigma_X)$, we can get a graded group homomorphism 
$$
\mathsf{PP}(\Sigma_X) \to \mathsf {CH}_*(X)
$$
by 
$$
F \mapsto [X] \cap F
$$
even for singular $X$. We continue to write $\mathsf{PP}(X)$ for the image of $\mathsf{PP}(X)$ in $\mathsf{CH}(X)$ in the general case as well. Since the map $\mathsf{PP}(\Sigma_X) \to \mathsf{CH}(X)$ factors through $\mathsf{CH}(\mathcal{A}_X)$, we have 
$$
\mathsf{PP}(X) \subset \SH (X).
$$
Equality holds whenever $X$ is smooth, or, more generally, quasi-smooth -- meaning $\Sigma_X$ is simplicial. In particular, we have 

\begin{corollary}
\label{cor: pullback}
Let $p:X' \to X$ be a log alteration and assume that $X$ is smooth. Then 
$$
p^*\SH X \subset \SH X'.
$$
\end{corollary}
\begin{proof}
This follows from the isomorphism $\SH X = \mathsf{PP}(X)$ and the fact that $\mathsf{PP}(X)$ is functorial with respect to Gysin pullback.  
\end{proof}

\begin{remark}
We note that in corollary \ref{cor: pullback}, we have that 
$$
p^*\mathsf{PP}(X) \subset \mathsf{PP}(X')
$$
for arbitrary $X$. However, the pullback on $\SH X$ does not make sense for arbitrary $X$, as $\SH X$ is a subgroup of Chow homology. 
\end{remark}

\subsection{Normalizations of Strata} We will need a slight generalization of the notion of strata homology that applies to individual strata of $X$ as well. Recall (\cite[Section 5]{MPS}) that the closures of the strata of a toroidal embedding $(X,D)$ need not be normal, and have an associated monodromy group $G$. Given a strata closure $\overline{X}_\sigma \subset X$, we will write $\widetilde{X}_\sigma$ for its normalization, and $Y_\sigma$ for the associated $G$-torsor over $\widetilde{X}_\sigma$. We will call the various monodromy torsors $Y_\sigma$ the \underline{normalized strata} of $X$. This is a misnomer for many reasons; for one, even when no normalization and monodromy issues are present, the $Y_\sigma$ are strata \emph{closures}, not strata -- in particular, a normal stratum is not a normalized stratum unless its closed -- but we keep the convention to avoid writing ``monodromy torsor over the normalization of stratum closure" everywhere. 

As is the case for the strata themselves, the normalized strata $Y_\sigma$ are also base changed from the Artin fan: we have a Cartesian diagram 
\[
\begin{tikzcd}
Y_\sigma \ar[r,"\alpha"] \ar[d] & \mathcal{Y}_\sigma \ar[d] \\ 
\overline{X}_\sigma \ar[r] \ar[d] & \overline{\mathcal{X}}_\sigma \ar[d] \\ 
X \ar[r] & \mathcal{A}_X
\end{tikzcd}
\]
where $\mathcal{Y}_\sigma$ is the monodromy torsor over the normalization of the closure $\overline{\mathcal{X}}_\sigma$ of the stratum $\mathcal{X}_\sigma$. 

\begin{definition}
Let $Y_\sigma$ be a normalized stratum. Its strata homology is 
$$
\SH (Y_\sigma) := \alpha^*\mathsf{CH}(\mathcal{Y}_\sigma)
$$
and its ring of piecewise polynomials is 
$$
\mathsf{PP}(Y_\sigma) := \alpha^* \mathsf{CH}_{\textup{op}}(\mathcal{Y}_\sigma). 
$$
\end{definition}
The argument showing that $\SH (X)$ is generated inductively by piecewise polynomials on strata and strata homology groups of deeper strata is easily adapted to strata homology for normalized strata as well. We note that log alterations of $X$ pull-back to log alterations of $Y_\sigma$, but are no longer birational, although they will still be proper and surjective. Nevertheless, the results \ref{cor:pushforward}, \ref{cor: pullback} still hold, with the same proof. We note
\begin{lemma}
\label{lem: ppsurjective}
The group $\SH X$ is generated by the push-forwards of the groups $\SH Y_\sigma$, as a $Y_\sigma$ ranges through the normalized strata of $X$. 
\end{lemma}

\begin{proof}
This is immediate from the excision sequence and the observation that since $\mathcal{Y}_\sigma \to \overline{\mathcal{X}}_\sigma$ is proper and birational, 
$$
\mathsf{CH}(\mathcal{Y}_\sigma) \to \mathsf{CH}(\overline{\mathcal{X}}_\sigma)
$$
is surjective. 
\end{proof}

We remark however that the resulting presentation is highly redundant.  

\begin{example}[Piecewise polynomials on $\Mgn$]
Let $S = \Mgn$, and $X = \overline{\mathcal{M}}_{g,n+1}$ be its universal curve. Write $f$ for the map $X \to S$. Strata of $S$ correspond to stable graphs of type $(g,n)$. We write $v$ for the vertices of $\Gamma$, which correspond to components $X_v$ of $X \to S$, $e$ for its edges, which correspond to nodes $X_e$, $h$ for its half edges, and $l$ for its $n$ legs, which correspond to the $n$ markings. Then the normalized strata of $\Mgn$ are precisely the moduli spaces 
$$
\overline{\mathcal{M}}_{\Gamma} = \prod_{v} \overline{\mathcal{M}}_{g(v),\textup{val}(v)}
$$
where $g(v)$ is the genus of the vertex $v$ and $\textup{val}(v)$ is its valence. The codimension of the stratum equals the number of edges $E$. If $i_\Gamma: \overline{\mathcal{M}}_{\Gamma} \to \Mgn$ denotes the natural map, the piecewise polynomials of $\Mgn$ are precisely the ring generated by 
$$
(i_\Gamma)_*F(c_1(N_1),\cdots, c_1(N_E))
$$
where $N_i$ are the direct summands the normal bundle of $i_\Gamma$, which splits on $\overline{\mathcal{M}}_{\Gamma}$, and $F$ is any $G := \textup{Aut}(\Gamma)$-invariant polynomial. In fact, the summands correspond to the edges $e$ of $\Gamma$, with  
$$
N_e = T_{h} \otimes T_{h'}
$$
-- the tensor product of the normal bundles of the node corresponding to the edge $e$ in the two components that contain it, which is precisely the data encoded by the two half-edges $h,h'$ whose union is $e$. The standard notation for this term is 
$$
(i_\Gamma)_*F(c_1(N_1),\cdots,c_1(N_E)) = [\Gamma, F(\psi_{h_1}+\psi_{h'_1},\cdots,\psi_{h_E}+\psi_{h'_E})]
$$
The strata of $\overline{\mathcal{M}}_{g,n+1}$ are in bijection with pairs $(\Gamma, c)$, where $c$ is either a vertex or an edge of $\Gamma$. Geometrically, pairs $(\Gamma,v)$ correspond to components of the relative curve $X \to S$, and $(\Gamma,e)$ with nodes. The normal bundles of the components $X_v$ in $X$ are simply the pullbakcs of the the normal bundles of $\overline{\mathcal{M}}_{\Gamma}$ in $\Mgn$. The normal bundles of the nodes $N_e$ have rank $1$-higher: the summand $N_e$ corresponding to the node $X_e$ is replaced by 
$$
T_{h} \oplus T_{h'}
$$
The pushforward of $c_1(T_{X_h})$ (resp. $c_1(T_{X_{h'}})$ all the way to $\Mgn$ is the term denoted by 
\[
[\Gamma,\psi_h] (\textup{ (resp. } [\Gamma,\psi_{h'}])
\]
reflecting the equality 
$$
c_1(N_e) = c_1(T_h)+c_1(T_{X_{h'}})
$$
Thus, the two classes $[\Gamma, \psi_h]$, $[\Gamma, \psi_{h'}]$ are \emph{not}  piecewise polynomials on $\Mgn$, but their sum is. Additional classes on $\overline{\mathcal{M}}_{\Gamma}$ are also of central interest: first, the restriction of the class of the log canonical $\omega_{X/\Mgn}$ to $X_v$; the pushforward of its $a+1$st power to $\Mgn$ is denote by 
\[
[\Gamma,\kappa_a^v]\footnote{Often reference to $v$ is left implicit, but is important for us to make explicit.}
\]
The push-forward of the square of the divisor $D_i$ corresponding to the $i$-th marking is on the other hand denoted by 
\[
[\Gamma,\psi_i]
\]
The classes $[\Gamma,\kappa_a^v],[\Gamma,\psi_i],[\Gamma,\psi_h],[\Gamma,\psi_{h'}]$ together generate the \emph{tautological ring} of $\Mgn$. We note that the classes split into two types: pushforwads of piecewise polynomials from normalized strata of $\overline{\mathcal{M}}_{g,n+1}$, and classes that are obtained by restricting classes from the total space of the universal family $\overline{\mathcal{M}}_{g,n+1}$ (rather than from individual normalized strata) to normalized strata, and pushing forward.  
\end{example}

\section{The Logarithmic Tautological Ring}

Our starting point in this section is a genus $g$, $n$-marked stable logarithmic curve $f: X \to S$: this means a log smooth morphism which is flat with geometrically reduced one dimensional fibers, which is Deligne-Mumford stable. We will also assume that $X$ and $S$ are smooth. The main and most interesting example is of course $X = \overline{\mathcal{M}}_{g,n+1},S=\Mgn$, but if anything, the added generality simplifies the notation. We define the tautological rings $R^\star(S),R^\star(X)$ of $S$ and $X$ to be the pull-back of the tautological ring from $\Mgn,\overline{\mathcal{M}}_{g,n+1}$. Furthermore, every normalized stratum $T$ of $S$, or $Y$ of $X$, map to corresponding normalized strata of $\Mgn,\overline{\mathcal{M}}_{g,n+1}$, and so we can define $R^\star(T),R^\star(Y)$ by pullback once more. We note that if $Y$ is a normalized stratum of $X$ dominating a normalized stratum $T$ of $S$, the restriction of $f$ to $Y \to T$ is either an isomorphism or a family of curves with smooth total space. We will keep writing $f$ for the restriction when the context is clear.   

Just as for $\Mgn$, we can write a generating set for $R^\star(S)$ in terms of the dualizing sheaf of $X \to S$, or the markings on $X$. We will however not do so; our point of view in this paper is that to prove theorem \ref{thm: main}, detailed knowledge of $R^\star$ is not required, but rather its formal properties, namely:

\underline{Data:} For every normalized stratum $Y$ of $X$ dominating a normalized stratum $T$ of $S$, graded subrings 
\begin{itemize}
\item $R^\star(T) \subset \mathsf{CH}(T)$ 
\item $R^\star(Y) \subset \mathsf{CH}(Y)$
\end{itemize}
satisfying \label{axioms} \underline{properties:}
\begin{enumerate}
\item $\mathsf{PP}(T) \subset R^\star(T)$, $\mathsf{PP}(Y) \subset R^\star(Y)$. 
\item $f_*R^\star(Y) = R^\star(T)$. 
\item $f^*R^\star(T) \subset R^\star(Y)$. 
\item The Chern class $c_1(T_f)$ of the relative tangent bundle are in $R^\star(Y)$.
\item The rings $R^\star(T)$ can be generated by $\mathsf{PP}(Y)$ and classes in $R^\star(X)$ as $i_Y: Y \to X$ ranges through all normalized strata of $X$ dominating $T$, by restricting and pushing forward, i.e.  
$$
R^\star(T) = \left \langle f_*(\mathsf{PP}(Y)), f_*(i_Y)^*R^\star(X) \right \rangle.
$$
\end{enumerate}
\begin{remark}
\label{remark: compatibility}
The last property is the most exotic, but powerful. A consequence is the much more intuitive compatibility of the system with respect to inclusions of normalized strata, in the following sense: for every morphism $i: P \to T$ of normalized strata, we have $i_*R^{\star}(P) \subset R^\star(T)$ and $i^*R^\star(T) \subset R^\star(P)$. To see this, choose normalized strata $Z$,$Y$ dominating $P,T$ and fitting into a commutative diagram 
\[
\begin{tikzcd}
Z \ar[d,"f|_Z"] \ar[r,"j"] & Y \ar[d,"f|_Y"] \\
P \ar[r,"i"] & T.
\end{tikzcd}
\]
Since $R^\star(P)$ is generated by classes of the form $\beta=(f|_{Z})_*F \cdot (i_Z)^*\gamma$ for $F \in \mathsf{PP}(Z), \gamma \in R^\star(X)$, we see that $i_*\beta = (f|_{Y})_*((j_*F)\cdot (i_Y)^*\gamma \cdot [Z])$. Since $j_*$ maps piecewise polynomials to piecewise polynomials, and the fundamental class of $Z$ is a piecewise polynomial on $Y$, we have that $(j_*F)\cdot (i_Y)^*\gamma \cdot [Z]$ is in $R^\star(Y)$, and so $i_*\beta \in R^\star(B)$. 

The analogous compatibility for maps of strata of $X$ follows by applying the same argument to the universal family over $X$!
\end{remark}

Intuitively, the last property says that the rings $R^\star(T)$ do not contain truly new classes relative to $R^\star(S)$: only those that are forced by the log geometry when restricting the map $X \to S$ to $T$. In effect, the data of the rings $R^\star(T),R^\star(Y)$ for normalized strata are redundant: we can recover everything from $R^\star(X)$ and $R^\star(S)$, and the map $X \to S$, together with the usual operations of intersection theory.

Let now $p:X' \to X$ be a log alteration. For every stratum $Y'$ of $X'$, surjecting onto a stratum $Y$ of $X$. We define 
\begin{definition}
\label{def: lift}
\begin{align*}
R^\star(Y') = \left \langle p^*R^\star(Y), \mathsf{PP}(Y') \right \rangle \\
R_\star(Y') = \left \langle p^*R^\star(Y),\SH (Y')\right \rangle
\end{align*}
to be the system generated by the pullbacks of the tautological ring on $X$ and piecewise polynomials or strata homology classes. 
\end{definition}
We define rings $R^\star,R_\star$ for log alterations $S' \to S$ similarly. We note that $R^\star$ is a ring, analogous to the operational theory, and $R_\star$ is a module over $R^\star$. When $X',S'$ are smooth, evaluation with the fundamental class yields an isomorphism between $R^\star$ and $R_\star$. In general, the quotient is isomorphic to $\SH X/\mathsf{PP}(X)$ and thus, of combinatorial nature. 

\begin{lemma}
\label{lem:pullbackpushforward}
Let $p:X' \to X$ be a log alteration. For every normalized stratum $Y'$ mapping to $Y$, we have 
\begin{align*}
p_*R_\star(Y') = R_\star(Y) \\
p^*R^\star(Y) \subset R^\star(Y').
\end{align*}
\end{lemma}

\begin{proof}
This follows from the projection formula and \ref{cor:pushforward},\ref{cor: pullback}. 
\end{proof}

The analogous statement for $S' \to S$ also holds of course, with the same proof.  

\begin{lemma}
\label{lem: tangent}
Let $p:S' \to S$ and $r:X' \to X \times_S S'$ be log alterations such that the composed map $g: X' \to S'$ is a family of curves, with $X'$ and $S'$ smooth. Then the Chern classes of $T_g$ are in $R^*(X')$.  
\end{lemma}

\begin{proof}
Write $q$ for the composition $X' \to X \times_S S' \to X$. We have a Cartesian diagram 
\[
\begin{tikzcd}
X' \ar[r] \ar[d] & \mathcal{A}_X' \ar[d] \\ 
X \ar[r,"\alpha_X"] & \mathcal{A}_X
\end{tikzcd}
\]
and $\alpha_X$ is smooth. So the relative tangent bundle $T_{X'/X}$ is pulled back from $\mathcal{A}_{X}'$. Similarly, the relative tangent bundle $T_{S'/S}$ is pulled back from $\mathcal{A}_{S'}$. Thus, their Chern classes are piecewise polynomials. As we have exact sequences 
\[
\begin{tikzcd}
0 \ar[r] & T_{X'/X} \ar[r] & T_{X'/S} \ar[r] & q^*T_{X/S} \ar[r] & 0 
\end{tikzcd}
\]
and
\[
\begin{tikzcd}
0 \ar[r] & T_{X'/S'} \ar[r] & T_{X'/S} \ar[r] & g^*T_{S'/S} \ar[r] & 0,
\end{tikzcd}
\]
we find that the Chern classes of $T_{X'/X}$ differ from the Chern classes of $q^*T_{X/S}$ by piecewise polynomials. Since $q^*c(T_{X/S}) \in R^\star(X')$, the claim follows.  
\end{proof}
The same argument shows the analogous statement for any map $Y \to T$ between normalized strata.

\begin{lemma}
\label{lem: generation}
Let $p: S' \to S$ be a log alteration, and write $r:X':= X \times_S S' \to X$ for the pullback of $X$. Then, for every normalized stratum $T' \to S'$,
$$
R_\star(T') \subset \left \langle f_*\SH(Y'), f_*i_{Y'}^*R^\star(X') \right \rangle 
$$
where $i_{Y'}: Y' \to X'$ ranges through normalized strata of $X'$ dominating $T'$.  
\end{lemma}
\begin{proof}
Suppose $\beta$ is in $R_\star(T')$. By definition, it has the form $p^*\gamma \cdot F$, for $\gamma \in R^\star(T)$, and $F$ in $\SH(T')$. But $\gamma$ is of the form $f_*(i_{Y}^*\delta \cdot \epsilon)$ for a normalized stratum $i_Y: Y \to X$ dominating $T$, $\delta$ a class on $X$, and $\epsilon \in \mathsf{PP}(Y)$. We pull-back $Y$ to $T'$, and get a diagram 

\[
\begin{tikzcd}
    X' \ar[r,"r"] & X \\
    Y' \ar[r,"q"] \ar[u,"i_{Y'}"] \ar[d,"g"] & Y \ar[u,"i_Y"] \ar[d,"f"] \\
    T' \ar[r,"p"] & T.
\end{tikzcd}
\]
Since we have $g_*q^*=p^*f_*$, it follows that 
\begin{align*}
\beta & = p^*(f_*(i_Y^*\delta \cdot \epsilon))\cdot F = g_*(q^*i_Y^*\delta \cdot q^*\epsilon) \cdot F \\ & = g_*(i_{Y'}^*r^*\delta \cdot q^*\epsilon) \cdot F = g_*(i_{Y'}^*r^*\delta \cdot q^*\epsilon \cdot g^*F).
\end{align*}
and $r^*\delta \in R(X')$, $q_*\epsilon \cdot g^*F \in \SH Y'$. 
 
\end{proof}

In particular, as in Remark \ref{remark: compatibility} we obtain that the rings $R_\star(T')$ are still compatible with respect to inclusions of normalized strata, in the sense that for every inclusion $i:P' \to T'$,  
\begin{align*}
i_*R_\star(P') \subset R_\star(T') \\
i^*R^\star(P) \subset R^\star(P').
\end{align*}
The analogous compatibility for the strata of log alterations of $X'$ comes from considering $X$ as the base of its own universal family, as above. We have written 
$$
R_\star(T') \subset \left \langle f_*\SH(Y'),f_*(i_{Y'})^*R(X') \right \rangle 
$$
because currently we do not know that the right hand side maps into the left hand side. The statement does however hold:

\begin{lemma}
Let $p:S' \to S$ be a log alteration, and $r: X' \to X \times_S S'$ a log alteration, such that $g: X' \to S'$ is a curve. Then $g_*R_\star(X') \subset R_\star(S')$ and $g^*R^\star(S') \subset R^\star(X')$.
\end{lemma}

\begin{proof}
The $g^*$ part is straightforward. Note that if 
\[
\begin{tikzcd}
X'' \ar[r,"\tau"] \ar[d,"h"] & X' \ar[d,"g"] \\ 
S'' \ar[r,"\sigma"] & S'
\end{tikzcd}
\]
is a diagram with $\sigma$ and $\tau$ log alterations, and the result holds for $h$, it also does for $g$, as $\tau_*R_\star(X'') \subset R_\star(X')$. Thus, by the semistable reduction theorem, we can assume $X'$ and $S'$ are smooth. 

Write $q$ for the composition $X' \to X$. We need to show that 
$$
g_*(\beta \cdot q^*\gamma) \in R_\star(S')
$$
where $\beta \in \mathsf{PP}(X')$ and $\gamma \in R(X)$. We can write $\beta$ as a sum of terms 
$$
\sum (i_{Y'})_* \beta_{Y'}
$$
where $Y'$ is a normalized stratum of $X'$, and $\beta_{Y'}$ is a piecewise polynomial on $Y'$. The map $Y'$ to $X$ will factor through some normalized strata $Y,T',T$ of $X,S',S$ respectively. Furthermore, since restrictions of tautological classes to $Y$ are tautological, we are reduced to proving the statement for the diagram of normalized strata

\[
\begin{tikzcd}
 Y' \ar[d,"g"] \ar[r,"q"] & Y\ar[d,"f"] \\ T' \ar[r,"p"] & T.
\end{tikzcd}
\]
The diagram falls into one of the three cases: 
\begin{itemize}
\item $Y'$ is a relative curve over $T'$, mapping to a relative curve $Y$ over $T$. 
\item $Y'$ is a relative node over $T'$, mapping to a relative node $Y$ over $T$. 
\item $Y'$ is a relative curve over $T'$, mapping to a relative node $Y$ over $T$. 
\end{itemize}

In the first two cases, the diagram is Cartesian, $p$ and $q$ are local complete intersection morphisms of the same dimension, and $f$ and $g$ are flat. Therefore, we have
$$
\mathsf{PP}(Y') = \mathsf{PP}(Y) \otimes_{\mathsf{PP}(T)} \mathsf{PP}(T')
$$
and so $\beta = g^*\delta \cdot q^*\epsilon$, for $\delta \in \mathsf{PP}(T') \subset R^\star(T')$ and $\epsilon \in \mathsf{PP}(Y) \subset R^\star(Y)$. Thus
$$
g_*(\beta \cdot q^*\gamma) = g_*(g^*\delta \cdot q^*(\epsilon \cdot \gamma) = \delta \cdot  p^*f_*(\epsilon \cdot \gamma)
$$
and $f_*(\epsilon \cdot \gamma) \in R(T)$ by the tautological ring properties. Therefore, the claim follows. In the third case, the diagram is no longer Cartesian, and $Y'$ is a $\mathbb{P}^1$ bundle over $T'$. As $X' \to X \times_S S'$, $S' \to S$ are logarithmic modifications, the $\mathbb{P}^1$-bundle is in fact pulled back from a $\mathbb{P}^1$-bundle of the corresponding Artin fans. Thus, its Chern classes are piecewise polynomials, and we have that its Chow group is generated by the hyperplane class $H$, with the relation
$$
\mathsf{CH}(Y') = \mathsf{CH}(T')[H]/H = \ell
$$
for a piecewise linear function $\ell \in \mathsf{PP}(T')$\footnote{In tropical terms, the stratum $Y'$ is obtained by subdividing an edge $e$ of length $\ell_e$ into two pieces, $\ell_1,\ell_2$. The piecewise linear function $\ell$ is $\ell_2-\ell_1$ .}. Thus
$$
\beta = F(H) \cdot g^*\delta
$$
for a polynomial $F$ in $H$ and $\delta \in \mathsf{PP}(T')$. Furthermore, $f$ is an isomorphism, and we find 
$$
q^*(\gamma) = g^*p^*(f_*(\gamma)).
$$
Thus 
$$
g_*(\beta \cdot q^*\gamma) = p^*f_*(\gamma) \cdot \delta \cdot F(\ell) \in R(T').
$$
By the compatibility property, the image of $g_*(\beta \cdot q^*\gamma)$ in $\mathsf{CH}(S')$ is in $R^\star(S')$. 
\end{proof}

The argument in fact shows a little more: for any normalized stratum $Y' \subset X'$ dominating a normalized stratum $T' \subset T$, we have $g_*R^\star(Y') \subset R^\star(T')$.

Therefore, combining with lemmas \ref{lem: tangent}, \ref{lem: generation}, we have shown 

\begin{corollary}
\label{cor: Axioms}
Let $p:S' \to S$, $r:X' \to X \times_S S'$ be log alterations keeping $X' \to S'$ a curve, with $X',S'$ smooth. Then the rings $\{(R^\star(T'),R^\star(Y')\}$ defined in \ref{def: lift} also satisfy the properties \ref{axioms} that the tautological ring does. 
\end{corollary}

In particular, the whole system is determined by $R^\star(X')$ and the various piecewise polynomials on strata.

\begin{definition}
Let $f:X \to S$ be a log curve with $X,S$ smooth. The \emph{logarithmic tautological ring} of $f$ is\footnote{Our terminology here clashes with the terminology we adopted in \cite{MPS}. The logarithmic tautological ring of loc. cit consists only of the piecewise polynomials, and is contained in the logarithmic tautological ring defined here.} 
\begin{align*}
\log R^*(S) = \varinjlim R^*(S')\\
\log R^*(X) = \varinjlim R^*(X').
\end{align*}
\end{definition}

\section{Grothendieck-Riemann-Roch and Twisting} 
\begin{lemma}
Let $q:X' \to X$ be a log alteration, and $\alpha$ a piecewise linear function on $X'$. Then the Chern classes of 
$$
Rq_*\mathcal{O}(-\alpha)
$$ 
are piecewise polynomials on $X$.
\end{lemma}

\begin{proof}
From
\[
\begin{tikzcd}
X' \ar[r] \ar[d] & \mathcal{A}_X' \ar[d] \\ 
X \ar[r,"\alpha_X"] & \mathcal{A}_X
\end{tikzcd}
\]
we deduce that $Rq_*\mathcal{O}(-\alpha)$ is pulled back from the Artin fan, and thus so are its Chern classes.
\end{proof}

\begin{lemma}
Let $f:X \to S$ be a log curve with $X$ and $S$ smooth. Let $p: S' \to S$ and $r: X' \to X \times_S S'$ be log alterations, with $X',S'$ smooth, so that the composition $g:X' \to S'$ remains a curve. Then the Todd class $\mathsf{Td}(g)$ is in $R^*(X')$.
\end{lemma}

\begin{proof}
The situation fits into the diagram 
\[
\begin{tikzcd}
X' \ar[d,"r"] \ar[rd,"q"]& \\ X \times_S S' \ar[d,"f'"] \ar[r,"p'"] & X \ar[d,"f"] \\ S' \ar[r,"p"] & S.
\end{tikzcd}
\]
$$
\mathsf{Td}(p \circ g) = \mathsf{Td}(f \circ q) 
$$
and so 
$$
\mathsf{Td}(g) = r^*\mathsf{Td}(f) \mathsf{Td}(q) g^*\mathsf{Td}(p)^{-1}.
$$
Since the Chern classes of $T_f$ are axiomatically in $R^*(X)$, the Todd class $\mathsf{Td}(f)$ is in $R_*(X)$. Similarly, since $\mathsf{Td}(p),\mathsf{Td}(q)$ are polynomial in the Chern classes of the relative tangent bundles $T_p,T_q$, the result follows from lemma \ref{lem: tangent}.  
\end{proof}

\begin{remark}
If $X', S'$ are not smooth, it is still possible to talk about a homological Todd class in $\mathsf{CH}(X')$ \cite[Section 18]{FultonInt}. This can be obtained as pushforward of the usual Todd class from any desingularization. Thus, by performing semistable reduction to $X' \to S'$, and pushing forward the resulting Todd class of the semistable family $X'' \to S''$, we find that it is still true that $\mathsf{Td}(g) \in R_\star(X')$. Since the map $g$ is lci, $\mathsf{Td}(g)$ is in fact the Todd class of $T_g$, and so in $R^\star(X')$. 
\end{remark}

\begin{theorem}
\label{thm: chernclassestaut}
Let $X \to S$ be as above, $\mathcal{L}$ a line bundle with $c_1$ in $R^*(X)$. Let 
\[
\begin{tikzcd}
X' \ar[d,"r"] \ar[rd,"q"]& \\ X \times_S S' \ar[d,"f'"] \ar[r,"p'"] & X \ar[d,"f"] \\ S' \ar[r,"p"] & S
\end{tikzcd}
\]
be a diagram with $p,r$ log alterations, with $g=f' \circ r$ a relative curve. Consider the line bundle $\mathcal{L}(-\alpha) := q^*\mathcal{L} \otimes \mathcal{O}(-\alpha)$, for $\alpha$ a piecewise linear function on $X'$. Then all Chern classes of 
$$
Rg_*\mathcal{L}(-\alpha)
$$ 
are in $R^*(S')$. 
\end{theorem}

\begin{proof}

By the projection formula, we have 

$$
Rg_* \mathcal{L}(-\alpha) = Rf'_*(p'^*\mathcal{L} \otimes Rr_*\mathcal{O}(-\alpha))
$$
and so by Grothendieck Riemann Roch we find 
$$
\mathrm{ch}(Rg_*\mathcal{L}(-\alpha)) = f'_*(p'^*\mathrm{ch}(\mathcal{L})\mathrm{ch}(Rr_*\mathcal{O}(-\alpha))p'^*\mathsf{Td}(f)) .
$$
Thus, the Chern character of $Rg_*\mathcal{L}(-\alpha)$ has the form $f'_*(p'^*(\alpha)\beta)$, where $\alpha \in R^*(X)$ and $\beta \in R_*(X \times_S S')$. Thus, it is in $R_*(X \times_S S')$, and thus its pushforward is in $R_*(S')$. We thus find that 
$$
c_i(Rg_*\mathcal{L}(-\alpha)) \cap [S'] \in R_*(S')
$$
and therefore $c_i(Rg_*\mathcal{L}(-\alpha)) \in R^*(S')$. 
\end{proof}

\section{Pullbacks of Universal Brill-Noether Classes}
We now fix a number of markings $n \ge 1$, a generic stability condition $\theta$, and data $A$ consisting of an integer $k$, a vector of integers $(a_1,\cdots,a_n)$, and a boundary divisor $D$ on the universal curve $\overline{\mathcal{C}}_{g,n}$. We look at the rational map
\begin{align*}
\mathsf{aj}_A: \Mgn \to \mathsf{Pic}_{g,n}^\theta\\
(C,x_1,\cdots,x_n) \mapsto \omega^k(\sum a_ix_i) \otimes \mathcal{O}(D|_C)
\end{align*}
and the classes 
$$
\mathsf{w}_{g,A,d}^r(\theta):= \mathsf{aj}_A^*(\mathsf{w}_{g,d}^r(\theta)) \in \mathsf{CH}^{g-\rho}(\Mgn)
$$
of the introduction. 

We apply the discussion of the previous section to $\overline{\mathcal{C}}_{g,n} \to \overline{\mathcal{M}}_{g,n}$ with $R$ the tautological rings. For any stability condition $\theta$, we can find a log modification 
$$
p^\theta: \overline{\mathcal{M}}_{g,A}^\theta \to \overline{\mathcal{M}}_{g,n}
$$
for which $\mathsf{aj}_A$ extends to a morphism 
$$
\mathsf{aj}_A^\theta: \DMgn^\theta \to \mathsf{Pic}_{g,n}^\theta.\\
$$
The construction is given in \cite{HMPPS}[Theorem 23] with $D=0$, but can be easily modified to any $D$. A similar construction is given in \cite{AP}, also for $D=0$ and a special choice of stability condition. We define: 

\begin{definition}
\[
\log \mathsf{w}_{g,d,A}^r(\theta) = (\mathsf{aj}_A^\theta)^*(\mathsf{w}_{g,d}^r(\theta)) \in \mathsf{CH}(\DMgn^\theta) \subset \log \mathsf{CH}(\Mgn,\partial \Mgn).
\]
\end{definition}
Since we have 
$$
\mathsf{aj}_A^* = p^\theta_* \circ (\mathsf{aj}_A^\theta)^*
$$
to show that $\mathsf{w}_{g,d,A}^r(\theta)$ are tautological, it suffices to show that
$$
(\mathsf{aj}_A^\theta)^*(\mathsf{w}_{g,d}^r(\theta))
$$ 
lie in $R(\DMgn^\theta)$.  Now, the functor of points of $\DMgn^\theta$, while hard to describe on the category of schemes, is straightforward on the category of log schemes: a log map $S \to \DMgn^\theta$ is the data of a quasi-stable curve $C \to S$, and a piecewise linear function $\alpha$ on $C$ such that 
$$
\omega_{C/S}^k(\sum a_ix_i) \otimes \mathcal{O}_D \otimes \mathcal{O}(-\alpha)
$$
is $\theta$-stable. In particular, $\DMgn^\theta$ carries a universal (quasi-stable) curve, which we denote by $\mathcal{C}^\theta$, and a piecewise linear function $\alpha$ on $\mathcal{C}^\theta$. Put together, we have a Cartesian diagram 

\[
\begin{tikzcd}
\mathcal{C}^\theta \ar[d,"\rho"] \ar[r,"s"] & \mathcal{C}_{g,n}^\theta \ar[d,"\pi"] \\ \DMgn^\theta \ar[r,"\mathsf{aj}_A^\theta"] & \mathsf{Pic}_{g,n}^\theta
\end{tikzcd}
\]

and we have 
$$
\mathcal{L}^\theta: = s^*\mathcal{L} = \omega^k(\sum a_ix_i) \otimes \mathcal{O}(D) \otimes \mathcal{O}(-\alpha).
$$

We now note

\begin{lemma}\cite{IllFGA}
\label{lem: pullbackcomplex}
Let 
\[
\begin{tikzcd}
X' \ar[r,"q"] \ar[d,"g"] & X \ar[d,"f"] \\ S' \ar[r,"p"] & S 
\end{tikzcd}
\]
be a Cartesian diagram with $f$ proper and flat. Suppose $\mathcal{L}$ is a line bundle on $X$ (or, more generally, any perfect complex). Then 
$$
Rg_*q^*\mathcal{L} = p^*(Rf)_*\mathcal{L}.
$$
\end{lemma}

\begin{proof}
This follows by cohomology and base change because $\mathcal{L}$ is flat over $S$. Namely, we have 
$$
Rg_*Lq^*\mathcal{L} = Lp^*Rf_*\mathcal{L}
$$
in general when $X'$ is replaced by the derived fiber product; in our case, the derived fiber product is $X'$ itself. Furthermore, since $\mathcal{L}$ is a line bundle, 
$$
Lq^*\mathcal{L} = q^*\mathcal{L}
$$
and $Lp^*Rf_*\mathcal{L} = p^*Rf_*\mathcal{L}$ since properness of $f$ implies that $Rf_*\mathcal{L}$ is a perfect complex. 
\end{proof}

It follows in particular that the degeneracy loci of the universal bundle on $\mathcal{C}_{g,n}^\theta$ pull back to the degeneracy loci of $\omega^k(\sum a_ix_i) \otimes \mathcal{O}(D) \otimes \mathcal{O}(-\alpha)$, 
$$
(\mathsf{aj}_A^\theta)^*\Delta_{g-d+r}^{(r+1)}(-R\pi_*\mathcal{L}) = \Delta_{g-d+r}^{(r+1)}(-R\rho_* \omega^k(\sum a_ix_i) \otimes \mathcal{O}(D) \otimes \mathcal{O}(-\alpha)).
$$
Since the Chern classes of $\omega^k(\sum a_ix_i) \otimes \mathcal{O}(D)$ are tautological in $\overline{\mathcal{C}}_{g,n}$, by theorem \ref{thm: chernclassestaut}, it follows that all Chern classes of $R\rho_* \omega^k(\sum a_ix_i) \otimes \mathcal{O}(D) \otimes \mathcal{O}(-\alpha)$ are in $R^\star(\DMgn^\theta)$. Therefore, in particular, we obtain

\begin{theorem}
The classes $\log \mathsf{w}_{g,A,d}^r(\theta)$ lie in $\log R^*(\Mgn,\partial \Mgn)$. 
\end{theorem}

We thus obtain as a corollary 

\begin{theorem}
\label{thm: mainthm}
The classes $
\mathsf{w}_{g,A,d}^r(\theta)$ lie in the tautological ring. 
\end{theorem}

\begin{remark}
Let $\mathfrak{Pic}_{g,n}$ denote the universal Picard stack over genus $g$, $n$-marked prestable curves: its objects over $S$ are genus $g$, $n$-marked prestable curves $C \to S$ with a line bundle $L$ on C. Write $\pi: \mathfrak{C}_{g,n} \to \mathfrak{Pic}_{g,n}$ for the universal curve, and $\mathcal{L}$ for the universal line bundle. Then lemma \ref{lem: pullbackcomplex} can be interpreted as saying that the degeneracy locus 
$$
\Delta_{g-d+r}^{(r+1)} (-R\pi_* \mathcal{L})
$$
defines an operational class on $\mathfrak{Pic}_{g,n}$, i.e. an element of
$$
\mathsf{CH}_{\textup{op}}(\mathfrak{Pic}_{g,n}).
$$
For more on this perspective, see \cite{BHPSS}.
\end{remark}

\subsection{Relation with the Double Ramification Cycle} We now specialize to the case $d=0$, $r=0$ with $A$ given by $D=0$, and a $k$ and integers $a_i$ such that $\sum a_i = k(2g-2)$. We furthermore assume that $\theta$ is a small, generic perturbation of the (degenerate) $0$ stability condition. 

\begin{lemma}
The class $\mathsf{w}_{0}^0(\theta) \in \mathsf{CH}(\mathsf{Pic}^\theta_{g,n})$ is the class of the $[0]$ section. 
\end{lemma}

\begin{proof}
This is \cite[Corollary 10]{KPH}.
\end{proof}

Now, recall the construction of the logarithmic double ramification cycle from \cite{Holmes}. We look at the rational Abel-Jacobi section 
$$
\mathsf{aj}_A: \Mgn \dasharrow \mathsf{Jac}_{g,n}.
$$
Holmes constructs a space $\Mdiamond$ over $\Mgn$ with an extension 
$$
\mathsf{aj}_A^{\diamond}: \Mdiamond \to \mathsf{Jac}_{g,n}.
$$
The space $\Mdiamond$ is universal with this property, in a sense that we do not make precise here, but it is not proper. Holmes proves that, on the other hand, the fiber product 
\[
\begin{tikzcd}
\mathsf{DRL}: = 
\Mdiamond \times_{\mathsf{Jac}} \Mgn \ar[r] \ar[d] & \Mgn \ar[d,"0"] \\
\Mdiamond \ar[r,"\mathsf{aj}_A"] & \mathsf{Jac}_{g,n}
\end{tikzcd}
\]
is proper. By Gysin pullback, one obtains a class
$$
\mathsf{aj}_A^{!}[0]
$$
supported on $\mathsf{DRL}$, which, since the latter is proper, can be pushed forward to $\Mgn$. The resulting pushforward is the double ramification cycle $\mathsf{DR}_{g,A}$. Holmes proves more: the non-proper space $\Mdiamond$ is an open in a (non-canonical), sufficiently fine log modification $M' \to \Mgn$. Thus, the class $\mathsf{aj}_A^{!}[0]$ can be pushed forward to $M'$. Let us provisionally call the resulting class $\mathsf{DR}_{g,A}(M') \in \mathsf{CH}(M')$. Certainly $\mathsf{DR}_{g,A}(M')$ pushes forward to $\mathsf{DR}_{g,A}$, and so determines it. Conversely, it is not true in general that the pullback of $\mathsf{DR}_{g,A}$ is $\mathsf{DR}_{g,A}(M')$, and so $\mathsf{DR}_{g,A}(M')$ contains strictly more information. 

Holmes proves that the classes $\mathsf{DR}_{g,A}(M')$, as $M'$ varies through log modifications of $\Mgn$ compactifying $\Mdiamond$ lift to the logarithmic Chow ring of $\Mgn$: If $p:M'' \to M'$ is a further modification, then 
$$
p^*\mathsf{DR}_{g,A}(M') = \mathsf{DR}_{g,A}(M'').
$$
Thus, the classes $\mathsf{DR}_{g,A}(M')$ can be considered all together as an element in 
$$
\log\mathsf{CH}(\Mgn, \partial \Mgn) = \varinjlim_{M' \to \Mgn} \mathsf{CH}(M')
$$
where $M' \to \Mgn$ ranges through all log modifications of $\Mgn$. The element is called the logarithmic double ramification cycle 
$$
\log \mathsf{DR}_{g,A}
$$
and the class $\mathsf{DR}_{g,A}(M')$ is called a \emph{representative} of $\log \mathsf{DR}_{g,A}$ on $M'$. 

What is important for us from the above discussion is that the space $\DMgn^\theta$ does compactify $\Mdiamond$, and so carries a representative $\mathsf{DR}_{g,A}(\DMgn^\theta)$ for $\log \mathsf{DR}_{g,A}$ (see \cite[Theorem A]{HMPPS}). We write $\mathsf{DR}_{g,A}^\theta$ for the class to lighten the notation. By construction, the space $\DMgn^\theta$ comes with a resolved Abel-Jacobi section 
$$
\mathsf{aj}_A^\theta: \DMgn^\theta \to \mathsf{Pic}_{g,n}^\theta.
$$

\begin{lemma}
Let $\theta$ be a small generic stability condition, and $A$ the discrete data consisting of an integer $k$ and a vector of integers $(a_1,\cdots,a_n)$ with $k(2g-2)=\sum a_i$. We have  
$$
\mathsf{DR}_{g,A}^\theta = (\mathsf{aj}_A^\theta)^*([0]) \in \mathsf{CH}(\DMgn^\theta)
$$.
\end{lemma}

\begin{proof}
The lemma has an elementary proof which is essentially a tautology, if one unwinds the definitions of $\DMgn^\theta$ and $\Mdiamond$. Since we do not wish to do so here \footnote{A thorough discussion of the functors of points will be given in forthcoming work with Melo,Ulirsch,Viviani and Wise.}, we give a roundabout, high-technology proof. It is shown in \cite{HMPPS} that 
$$
\mathsf{DR}_{g,A}^\theta = \mathsf{DR}^{\textup{op}}_{g,\emptyset,\mathcal{L}^\theta}.
$$
That is, $\mathsf{DR}_{g,A}^\theta$ coincides with the ``universal double ramification cycle" of \cite{BHPSS}, calculated for the line bundle $\mathcal{L}^\theta = \omega^k(\sum a_ix_i) \otimes \mathcal{O}(-\alpha)$ on $\mathcal{C}^\theta$. The latter is defined as the pullback to $\DMgn^\theta$ of the \emph{closure} of the $0$ section via the Abel-Jacobi section 
$$
\mathsf{aj}_A:\DMgn \to \mathfrak{Pic}_{g,n}
$$
where $\mathfrak{Pic}_{g,n}$ is the universal Picard stack: it parametrizes genus $g$, $n$-marked prestable curves with line bundles. Symbolically 
$$
\mathsf{DR}^{\textup{op}}_{g,\emptyset,\mathcal{L}^\theta}:=\mathsf{aj}_A^*[\overline{0}].
$$
The compactified Jacobian $\mathsf{Pic}_{g,n}^\theta$ is an open substack of $\mathfrak{Pic}_{g,n}$, and we have a commutative diagram
\[
\begin{tikzcd}
 & \mathsf{Pic}_{g,n}^\theta \ar[dd] \\
 \DMgn^\theta \ar[ur,"\mathsf{aj}_A^\theta"] \ar[dr,"\mathsf{aj}_A"] & \\
 & \mathfrak{Pic}_{g,n}.
\end{tikzcd}
\]
For a small generic stability condition $\theta$, the $0$ section is closed in $\mathsf{Pic}_{g,n}^\theta$. Hence 
$$
\mathsf{aj}_A^*([\overline{0}]) = (\mathsf{aj}_A^\theta)^*([0])
$$
and thus 
$$
\mathsf{DR}_{g,A}^\theta = (\mathsf{aj}_A^\theta)^*([0]).
$$
\end{proof}

We have thus shown 

\begin{corollary}
We have 
$$
\mathsf{DR}_{g,A}^\theta = w_{g,A,0}^0(\theta).
$$
\end{corollary}

Combining this with theorem \ref{thm: mainthm}, we obtain 

\begin{corollary}
The class $\mathsf{DR}_{g,A}^\theta \in R^*(\DMgn^\theta)$.
\end{corollary}

 In particular, all intersections of classes $\mathsf{DR}_{g,A}^\theta$ for various choices of $A,\theta$ push forward to tautological classes in $\Mgn$.

\bibliography{refs2}

\newcommand{\etalchar}[1]{$^{#1}$}
\providecommand{\bysame}{\leavevmode\hbox to3em{\hrulefill}\thinspace}
\providecommand{\MR}{\relax\ifhmode\unskip\space\fi MR }
\providecommand{\MRhref}[2]{%
  \href{http://www.ams.org/mathscinet-getitem?mr=#1}{#2}
}
\providecommand{\href}[2]{#2}
\begin{thebibliography}{KKMSD73}

\bibitem[AK00]{AK}
D.~Abramovich and K.~Karu, \emph{Weak semistable reduction in characteristic
  0}, Invent. Math. \textbf{139} (2000), no.~2, 241--273. \MR{1738451
  (2001f:14021)}

\bibitem[ALT18]{ALT}
Karim Adiprasito, Gaku Liu, and Michael Temkin, \emph{Semistable reduction in
  characteristic 0}, 2018.

\bibitem[AP21]{AP}
Alex Abreu and Marco Pacini, \emph{The resolution of the universal {A}bel map
  via tropical geometry and applications}, Adv. Math. \textbf{378} (2021),
  Paper No. 107520, 62. \MR{4184297}

\bibitem[AW18]{AW}
Dan Abramovich and Jonathan Wise, \emph{Birational invariance in logarithmic
  {G}romov-{W}itten theory}, Compos. Math. \textbf{154} (2018), no.~3,
  595--620. \MR{3778185}

\bibitem[Bar18]{BarrottChow}
Lawrence~Jack Barrott, \emph{Logarithmic chow theory}, 2018.

\bibitem[BHP{\etalchar{+}}21]{BHPSS}
Younghan Bae, David Holmes, Rahul Pandharipande, Johannes Schmitt, and Rosa
  Schwarz, \emph{Pixton's formula and abel-jacobi theory on the picard stack},
  2021.

\bibitem[BV12]{BV}
Niels Borne and Angelo Vistoli, \emph{Parabolic sheaves on logarithmic
  schemes}, Adv. Math. \textbf{231} (2012), no.~3-4, 1327--1363. \MR{2964607}

\bibitem[Cap94]{Cap}
Lucia Caporaso, \emph{A compactification of the universal {P}icard variety over
  the moduli space of stable curves}, J. Amer. Math. Soc. \textbf{7} (1994),
  no.~3, 589--660. \MR{1254134}

\bibitem[CCUW20]{CCUW}
Renzo Cavalieri, Melody Chan, Martin Ulirsch, and Jonathan Wise, \emph{A moduli
  stack of tropical curves}, Forum Math. Sigma \textbf{8} (2020), Paper No.
  e23, 93. \MR{4091085}

\bibitem[DSvZ21]{admcycles}
Vincent Delecroix, Johannes Schmitt, and Jason van Zelm, \emph{admcycles---a
  {S}age package for calculations in the tautological ring of the moduli space
  of stable curves}, J. Softw. Algebra Geom. \textbf{11} (2021), no.~1,
  89--112. \MR{4387186}

\bibitem[Est01]{Est}
Eduardo Esteves, \emph{Compactifying the relative {J}acobian over families of
  reduced curves}, Trans. Amer. Math. Soc. \textbf{353} (2001), no.~8,
  3045--3095. \MR{1828599}

\bibitem[FP05]{FabPan}
C.~Faber and R.~Pandharipande, \emph{Relative maps and tautological classes},
  J. Eur. Math. Soc. (JEMS) \textbf{7} (2005), no.~1, 13--49. \MR{2120989}

\bibitem[Ful98]{FultonInt}
William Fulton, \emph{Intersection theory}, second ed., Ergebnisse der
  Mathematik und ihrer Grenzgebiete. 3. Folge. A Series of Modern Surveys in
  Mathematics [Results in Mathematics and Related Areas. 3rd Series. A Series
  of Modern Surveys in Mathematics], vol.~2, Springer-Verlag, Berlin, 1998.
  \MR{1644323}

\bibitem[GM15]{GMtor}
W.~D. Gillam and Sam Molcho, \emph{A theory of stacky fans}, 2015.

\bibitem[GP03]{GrPannonTaut}
T.~Graber and R.~Pandharipande, \emph{Constructions of nontautological classes
  on moduli spaces of curves}, Michigan Math. J. \textbf{51} (2003), no.~1,
  93--109. \MR{1960923}

\bibitem[HKP18]{KPH}
David Holmes, Jesse~Leo Kass, and Nicola Pagani, \emph{Extending the double
  ramification cycle using {J}acobians}, Eur. J. Math. \textbf{4} (2018),
  no.~3, 1087--1099. \MR{3851130}

\bibitem[HMP{\etalchar{+}}22]{HMPPS}
D.~Holmes, S.~Molcho, R.~Pandharipande, A.~Pixton, and J.~Schmitt,
  \emph{Logarithmic double ramification cycles}, 2022.

\bibitem[Hol21]{Holmes}
David Holmes, \emph{Extending the double ramification cycle by resolving the
  {A}bel-{J}acobi map}, J. Inst. Math. Jussieu \textbf{20} (2021), no.~1,
  331--359. \MR{4205785}

\bibitem[HPS19]{HPS}
David Holmes, Aaron Pixton, and Johannes Schmitt, \emph{Multiplicativity of the
  double ramification cycle}, Doc. Math. \textbf{24} (2019), 545--562.
  \MR{3960120}

\bibitem[HS22]{HS}
David Holmes and Rosa Schwarz, \emph{Logarithmic intersections of double
  ramification cycles}, Algebr. Geom. \textbf{9} (2022), no.~5, 574--605.
  \MR{4490707}

\bibitem[Ill05]{IllFGA}
Luc Illusie, \emph{Grothendieck's existence theorem in formal geometry},
  Fundamental algebraic geometry, Math. Surveys Monogr., vol. 123, Amer. Math.
  Soc., Providence, RI, 2005, With a letter (in French) of Jean-Pierre Serre,
  pp.~179--233. \MR{2223409}

\bibitem[JPPZ17]{JPPZ}
F.~Janda, R.~Pandharipande, A.~Pixton, and D.~Zvonkine, \emph{Double
  ramification cycles on the moduli spaces of curves}, Publ. Math. Inst. Hautes
  \'{E}tudes Sci. \textbf{125} (2017), 221--266. \MR{3668650}

\bibitem[Kat89]{K}
Kazuya Kato, \emph{Logarithmic structures of {F}ontaine-{I}llusie}, Algebraic
  analysis, geometry, and number theory ({B}altimore, {MD}, 1988), Johns
  Hopkins Univ. Press, Baltimore, MD, 1989, pp.~191--224. \MR{1463703
  (99b:14020)}

\bibitem[KKMSD73]{KKMS}
G.~Kempf, Finn~Faye Knudsen, D.~Mumford, and B.~Saint-Donat, \emph{Toroidal
  embeddings. {I}}, Lecture Notes in Mathematics, Vol. 339, Springer-Verlag,
  Berlin, 1973. \MR{0335518 (49 \#299)}

\bibitem[Kon92]{KonAiry}
Maxim Kontsevich, \emph{Intersection theory on the moduli space of curves and
  the matrix {A}iry function}, Comm. Math. Phys. \textbf{147} (1992), no.~1,
  1--23. \MR{1171758}

\bibitem[KP19]{KP}
Jesse~Leo Kass and Nicola Pagani, \emph{The stability space of compactified
  universal {J}acobians}, Trans. Amer. Math. Soc. \textbf{372} (2019), no.~7,
  4851--4887. \MR{4009442}

\bibitem[Mel19]{Melo}
Margarida Melo, \emph{Universal compactified {J}acobians}, Port. Math.
  \textbf{76} (2019), no.~2, 101--122. \MR{4065093}

\bibitem[Mol21]{Mss}
Sam Molcho, \emph{Universal stacky semistable reduction}, Israel J. Math.
  \textbf{242} (2021), no.~1, 55--82. \MR{4282076}

\bibitem[MPS21]{MPS}
Samouil Molcho, Rahul Pandharipande, and Johannes Schmitt, \emph{The hodge
  bundle, the universal 0-section, and the log chow ring of the moduli space of
  curves}, 2021.

\bibitem[MR21]{MRan}
Sam Molcho and Dhruv Ranganathan, \emph{A case study of intersections on
  blowups of the moduli of curves}, 2021.

\bibitem[Mum83]{MumTow}
David Mumford, \emph{Towards an enumerative geometry of the moduli space of
  curves}, Arithmetic and geometry, {V}ol. {II}, Progr. Math., vol.~36,
  Birkh\"{a}user Boston, Boston, MA, 1983, pp.~271--328. \MR{717614}

\bibitem[OS79]{OdaSeshadri}
Tadao Oda and C.~S. Seshadri, \emph{Compactifications of the generalized
  {J}acobian variety}, Trans. Amer. Math. Soc. \textbf{253} (1979), 1--90.
  \MR{536936}

\bibitem[Pan96]{Pan}
Rahul Pandharipande, \emph{A compactification over {$\overline {M}_g$} of the
  universal moduli space of slope-semistable vector bundles}, J. Amer. Math.
  Soc. \textbf{9} (1996), no.~2, 425--471. \MR{1308406}

\bibitem[PP21]{PPrel}
R.~Pandharipande and A.~Pixton, \emph{Relations in the tautological ring of the
  moduli space of curves}, Pure Appl. Math. Q. \textbf{17} (2021), no.~2,
  717--771. \MR{4257600}

\bibitem[PPZ15]{PPZ}
Rahul Pandharipande, Aaron Pixton, and Dimitri Zvonkine, \emph{Relations on
  {$\overline{\mathcal{M}}_{g,n}$} via {$3$}-spin structures}, J. Amer. Math.
  Soc. \textbf{28} (2015), no.~1, 279--309. \MR{3264769}

\bibitem[PRvZ20]{PRvZ}
Nicola Pagani, Andrea~T. Ricolfi, and Jason van Zelm, \emph{Pullbacks of
  universal {B}rill-{N}oether classes via {A}bel-{J}acobi morphisms}, Math.
  Nachr. \textbf{293} (2020), no.~11, 2187--2207. \MR{4188687}

\bibitem[Ran19]{RanProduct}
Dhruv Ranganathan, \emph{A note on the cycle of curves in a product of pairs},
  2019.

\bibitem[Wit91]{Witten}
Edward Witten, \emph{Two-dimensional gravity and intersection theory on moduli
  space}, Surveys in differential geometry ({C}ambridge, {MA}, 1990), Lehigh
  Univ., Bethlehem, PA, 1991, pp.~243--310. \MR{1144529}

\end{thebibliography}

\vspace{8pt}

\noindent Departement Mathematik, ETH Z\"urich\\
\noindent Rämistrasse 101, 8092 Zürich, Switzerland\\
\noindent samouil.molcho@math.ethz.ch

\end{document}